\documentclass[12pt]{amsart}
\usepackage{latexsym,amssymb,amsmath,tikz}

\numberwithin{equation}{section}
\newtheorem{theorem}{Theorem}[section]
\newtheorem{lemma}[theorem]{Lemma}

\newtheorem{corollary}[theorem]{Corollary}

\theoremstyle{definition}
\newtheorem{definition}[theorem]{Definition} 
\newtheorem{remark}[theorem]{Remark}
\newtheorem{example}[theorem]{Example}

\begin{document}

%%%%%%%%%%%%%%%%%%%%%%%%%%%%%%%%%%%%%%%%%%%%%%%%%%%%%%%%%%%%%%%% 
 
\title{Betti numbers of toric ideals of graphs: A case study}
\thanks{Version: November 7, 2018}

\author[F. Galetto]{Federico Galetto}
\author[J. Hofscheier]{Johannes Hofscheier}
\author[G. Keiper]{Graham Keiper}
\author[C. Kohne]{Craig Kohne}
\author[M. Paczka]{Miguel Eduardo Uribe Paczka}
\author[A. Van Tuyl]{Adam Van Tuyl}

\address{Department of Mathematics\\
  Cleveland State University
  Cleveland, OH 44115-2215}
\email{f.galetto@csuohio.edu}

\address{Department of Mathematics and Statistics\\
McMaster University, Hamilton, ON, L8S 4L8}
\email{hofschej@mcmaster.ca, keipergt@mcmaster.ca, kohnec@math.mcmaster.ca,\newline 
vantuyl@math.mcmaster.ca}

\address{Department of Mathematics\\
Escuela Superior de F\'{i}sica y Matem\'{a}ticas\\
Instituto Polit\'{e}cnico Nacional\\
07300 Mexico City.} 
\email{muribep1700@alumno.ipn.mx}


\begin{abstract}
We compute the graded Betti numbers for the toric ideal of a 
family of graphs constructed by adjoining a cycle to a complete bipartite 
graph.   The key observation is that this family admits an initial ideal 
which has linear quotients.  As a corollary, we compute the 
Hilbert series and $h$-vector for all the toric ideals of graphs in this family.
\end{abstract}
 
\maketitle

%%%%%%%%%%%%%%%%%%%%%%%%%%%%%%%%%%%%%%%%%%%%%%%%%%%%%%%%%%%%%%%%%%%%%

\section{Introduction}

Let $G$ be any finite simple graph with edge set $E(G) = \{e_1,\ldots,e_q\}$ 
and vertex set $V(G) = \{x_1,\ldots,x_n\}$.  We define the polynomial rings 
$k[V(G)]=k[x_1,\ldots,x_n]$ and $k[E(G)]=k[e_1,\ldots,e_q]$ for any field $k$,  
and the ring homomorphism
\[\phi_G:k[E(G)] \rightarrow k[V(G)] \text{ by }
e_i \mapsto x_jx_k \text{ if } e_i = \{x_j,x_k\}.
\]
The kernel of the map $\phi_G$, denoted by $I_G = {\rm ker} (\phi_G)$,
is called the \textbf{toric ideal} associated with the graph $G$.  
It is well-known that  $I_{G}$ is a homogeneous ideal generated by binomials 
(see \cite[Prop. 5.19]{EH}), and in particular, 
the generators of $I_{G}$ correspond to closed even walks in $G$ 
(see \cite{V1}, and \cite{RTT,S} for a detailed analysis of the minimal
generators).

Associated with $I_{G} \subseteq R = k[E(G)]$  is a 
\textbf{minimal graded free resolution} 
of the form:
\[
0 \rightarrow \bigoplus_{j} R(-j)^{\beta_{l,j}(I_{G})} 
\rightarrow   \bigoplus_{j} R(-j)^{\beta_{l-1,j}(I_{G})} \rightarrow \cdots 
\rightarrow   \bigoplus_{j} R(-j)^{\beta_{0,j}(I_{G})} \rightarrow 
I_{G} \rightarrow 0
\]
where $l \leq q$ and $R(-j)$ is the free $R$-module obtained by shifting 
the degrees of $R$ by $j$ (i.e., so that $R(-j)_a = R_{a-j}$). 
The number $\beta_{i,j}(I_{G})$ is called the $i,j$-th 
\textbf{graded Betti number} of $I_{G}$ and equals the number of
minimal generators of degree $j$ in the $i$-th syzygy module of 
$I_{G}$.    Ideally, we would like to determine $\beta_{i,j}(I_G)$ directly from a description of $G$.  A non-exhaustive
list of papers that have worked towards a dictionary between the invariants
of the minimal resolution and $G$ includes \cite{BHO,BOVT,CN,D,GV,HHKO,HHKO2,K,
OH,RTT,V1}.
However, given that it can be difficult
to enumerate the minimal generators of $I_G$, i.e., the numbers
$\beta_{0,j}(I_G)$, a general solution to this problem is currently unknown.

Given the difficulty of finding $\beta_{i,j}(I_G)$ 
for arbitrary graphs, one is led to study restricted families. 
In this paper we carry out a case study of a special family of graphs, and we
show that for this family, we can determine all of the graded Betti numbers
of the corresponding toric ideal.  In particular, we focus on a family of 
graphs, denoted $G_{r,d}$, that is constructed from the complete 
bipartite graph $K_{2,d}$ with $d \geq 2$,
and joining the two vertices of degree $d$ 
with an even path of length $2r-2$ with $r\ge 3$ (see Figure \ref{example_graph} for 
the graph $G_{3,5}$). These graphs are especially interesting with regard to Stanley's conjecture which asks whether any graded Cohen-Macaulay domain has a unimodal $h$-vector. We refer to Remark \ref{rem:Stanley-conj} for further details and references. From this point of view the graphs $G_{r,d}$ are appealing 
because the associated Cohen-Macaulay domains $k[E(G)]/I_G$ have $h$-vector $(1,d, \ldots, d)$ (see Theorem \ref{h-vector-theorem}). The question now is whether one can modify the very tractable graphs $G_{r,d}$ in such a way that the associated Cohen-Macaulay domain is still normal, but its $h$-vector is no longer unimodal. The quest for such examples is an active area of research in Ehrhart theory and similar approaches have been already made (see \cite{BD}).

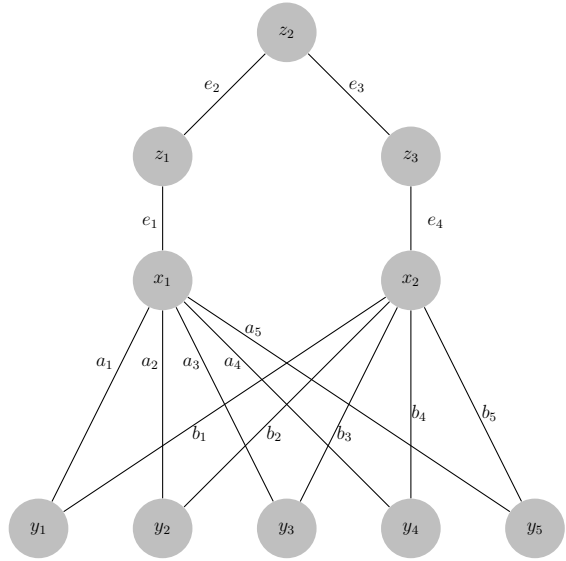
\begin{figure}
\centering
\begin{tikzpicture}
 [scale=.55,auto=left,every node/.style={circle,fill=lightgray, scale=.6, minimum size=3.2em}]
\node (x1) at (3,6) {\text{$x_{1}$}};
\node (x2) at (9,6) {\text{$x_{2}$}};
\node (y1) at (0,0) {\text{$y_{1}$}};
\node (y2) at (3,0) {\text{$y_{2}$}};
\node (y3) at (6,0) {\text{$y_{3}$}};
\node (y4) at (9,0) {\text{$y_{4}$}};
\node (y5) at (12,0) {\text{$y_{5}$}};
\node (z1) at (3,9) {\text{$z_{1}$}};
\node (z2) at (6,12) {\text{$z_{2}$}};
\node (z3) at (9,9) {\text{$z_{3}$}};

\node[fill=white] (a1) at (1.6,4) {\text{$a_{1}$}};
\node[fill=white] (a2) at (2.7,4) {\text{$a_{2}$}};
\node[fill=white] (a3) at (3.7,4) {\text{$a_{3}$}};
\node[fill=white] (a4) at (4.7,4) {\text{$a_{4}$}};
\node[fill=white] (a5) at (5.2,4.8) {\text{$a_{5}$}};
\node[fill=white] (b1) at (3.9,2.3) {\text{$b_{1}$}};
\node[fill=white] (b2) at (5.7,2.3) {\text{$b_{2}$}};
\node[fill=white] (b3) at (7.4,2.3) {\text{$b_{3}$}};
\node[fill=white] (b4) at (9.2,2.8) {\text{$b_{4}$}};
\node[fill=white] (b5) at (10.9,2.8) {\text{$b_{5}$}};
\node[fill=white] (e1) at (2.7,7.4) {\text{$e_{1}$}};
\node[fill=white] (e2) at (4.2,10.7) {\text{$e_{2}$}};
\node[fill=white] (e4) at (9.6,7.4) {\text{$e_{4}$}};
\node[fill=white] (e3) at (7.7,10.7) {\text{$e_{3}$}};

\draw (z1) -- (z2);
\draw (z3) -- (z2);

\draw (x1) -- (z1);
\draw (x1) -- (y1);
\draw (x1) -- (y2);
\draw (x1) -- (y3);
\draw (x1) -- (y4);
\draw (x1) -- (y5);

\draw (x2) -- (z3);
\draw (x2) -- (y1);
\draw (x2) -- (y2);
\draw (x2) -- (y3);
\draw (x2) -- (y4);
\draw (x2) -- (y5);

\end{tikzpicture}
\caption{Illustration of $G_{3,5}$}\label{example_graph}
\end{figure}
 
To compute the graded Betti numbers, we find
a monomial order such that the 
corresponding initial ideal $J = {\rm in}(I_{G_{r,d}})$
is a square-free monomial ideal.  We exploit the fact that
$\beta_{i,j}(I_{G_{r,d}}) \leq \beta_{i,j}(J)$, and that for all $G_{r,d}$, the 
ideal $J$ has linear quotients whose graded Betti numbers can be 
computed using \cite{SV}.   

The approach of using the initial ideal
was also employed in \cite{BOVT}; we hope this case study
further illustrates the usefulness of this technique to 
compute (or bound) graded Betti numbers of toric ideals, especially since  
we know of very few families of graphs where we know all the values
$\beta_{i,j}(I_G)$.  Our work is in the spirit of Conca
and Varbaro's paper \cite{CV} which shows that many of the homological
invariants of an ideal are related to its initial ideal if the initial ideal is
square-free. Note that our special family can give us some partial insight
into the graded Betti numbers of $I_G$ for more general graphs.  In
particular,  if $G_{r,d}$ is an induced subgraph of $G$,
we can apply work of
Beyarslan, H\`a, and O'Keefe \cite{BHO} 
to show $\beta_{i,j}(I_{G_{r,d}}) \leq \beta_{i,j}(I_G)$;
Theorem \ref{invariants}  uses this inequality
to give lower bounds on the Castelnuovo-Mumford
regularity and projective dimension for any toric ideal $I_G$.

We have organized the paper as follows. In Section~\ref{prelim-sec}, we 
recall the necessary preliminaries. In Section~\ref{betti-numbers-Grd}, we 
provide an explicit formula for the graded Betti numbers of the 
toric ideal of the graphs $G_{r,d}$ (see Theorem~\ref{betti-full}). 
Finally, in Section~\ref{consequences}, among other things, 
we compute the Hilbert series for the ring 
$k[E(G_{r,d})] \slash I_{G_{r,d}}$ (see Theorem~\ref{h-vector-theorem}).

\noindent
{\bf Acknowledgments.}  
This paper started as a class project for a graduate reading course
on toric ideals.  Paczka thanks CONACYT for financial support and 
the hospitality of McMaster University. Van Tuyl's
research was supported by NSERC Discovery Grant 2014-03898.
Experiments with
{\it Macaulay2} \cite{GS} led to many of our results.  We would
also like to thank Jennifer Biermann, Tina O'Keefe, and the referee for their 
feedback.

%%%%%%%%%%%%%%%%%%%%%%%%%%%%%%%%%%%%%%%%%%%%%%%%%%%%%%%%%%%%%%%%%%%

\section{Background}\label{prelim-sec}

We recall the relevant graph theory and commutative algebra background.

\subsection{Simple graphs, walks, and binomials}
Let $G$ be a graph with vertex set $V(G)=\{x_1,...,x_n\}$ and edge set 
$E(G)=\{e_1,...,e_q\}$. For each edge $e_i$ there is an associated pair of 
distinct vertices $\{x_{i_{1}},x_{i_{2}}\}$; we will use $e_i$ and 
$\{x_{i_{1}},x_{i_{2}}\}$ interchangeably.  All graphs are assumed to 
be  \textbf{finite simple graphs}, i.e., $|V(G)| < \infty$ and there is at most
one edge between any pair of vertices. 

A \textbf{walk} $W$ of length $r$ is a sequence of edges 
$W = (\{v_0,v_1\},\{v_1,v_2\},\dots,\{v_{r-1},v_r\})$, where the $v_j$ are 
vertices in $V(G)$ and the $\{v_j,v_{j+1}\}$ are the edges connecting $v_j$ 
and $v_{j+1}$. A walk is a \textbf{closed walk} if $v_0=v_r$, and it is 
an \textbf{even walk} if $r$ is even. 
We say a closed walk 
$W=(e_0,e_1,\ldots,e_{r-1})$ is a \textbf{minimal} closed walk if no two 
consecutive (modulo $r$) edges are equal. 

The generators of the toric ideal of $I_G$ are binomials that can
be described in terms of closed even walks in $G$. Specifically,  
if $W = (\{v_0,v_1\},\{v_1,v_2\},\dots,\{v_{2s-1},v_{2s}=v_0\})$ is a closed
even walk, 
we associate with $W$ the binomial
\begin{equation*}
f_W = \prod_{j=1}^{s}e_{2j-1} - \prod_{j=1}^{s}e_{2j} \in k[e_1,\dots,e_q]  
\end{equation*}
where $e_j = \{v_{j-1},v_{j}\}$ for $1\leq j \leq 2s$. A binomial $f = u - v \in I_{G}$ is called a \textbf{primitive binomial} if there is no other binomial $g = u' - v' \in I_{G}$ such that $u'|u$ and $v'|v$. A closed even walk $W$ of $G$ is a \textbf{primitive walk} if its associated binomial $f_{W}$ is a primitive binomial. The following observation will be useful in the sequel.
\begin{theorem}{\rm\cite[Proposition~5.19]{EH}} 
\label{primitive-theorem}
Let G be a finite simple graph. Then the set of primitive walks in $G$ is a Gr\"{o}bner basis of $I_{G}$ with respect to any monomial ordering.
\end{theorem}

The next lemma follows directly from the definitions.

\begin{lemma}\label{closed-car}\label{closed-descom}
Let $G$ be a finite simple graph.
\begin{enumerate}
\item[$(i)$]
All primitive walks of $G$ are minimal walks.
\item[$(ii)$] A closed even walk $W$ of $G$ which can be expressed as 
two consecutive closed even walks, i.e., 
\[
W = (\{x_{1,1},x_{1,2}\},\{x_{1,2},x_{1,3}\}, \ldots  ,\{x_{1,n},x_{1,1}\}   ,\{x_{2,1},x_{2,2}\},\{x_{2,2},x_{2,3}\}, \ldots  ,\{x_{2,m},x_{2,1}\})
\]
with $x_{1,1} = x_{2,1}$, is not primitive.
\end{enumerate}
\end{lemma}

\subsection{Linear quotients, initial ideals, and 
Hilbert series} 
We recall some general results about homogeneous ideals
$I$ in a polynomial ring $S = k[z_1,\ldots,z_n]$.  
Ideals with linear quotients, which play a key role in our main theorem, 
were introduced by Herzog and
Takayama \cite{HT};  a more general definition can be found in
Sharifan and Varbaro \cite{SV}.

\begin{definition}\label{quotient-defm}
Let $I \subseteq S$ be a homogeneous ideal with an ordered 
minimal set of homogeneous generators $\{f_1,\ldots,f_{m}\}$. Then $I$ has 
\textbf{linear quotients} with respect to $\{f_1,\ldots,f_{m}\}$ if each 
ideal quotient 
$\langle f_1,\ldots, f_{j-1}\rangle: \left\langle  f_j \right\rangle$ 
for $j=2,...,m$ is generated by linear forms. 
\end{definition}

When $I$ has linear quotients, the graded Betti numbers are a function of 
the number of generators of the ideal quotients.

\begin{theorem}[{\cite[Corollary~2.7]{SV}}]\label{quotient-compute}
Let $I \subseteq S$ 
be a homogeneous ideal with linear quotients with respect to 
$f_1,\ldots,f_m$ where $\{f_1,\ldots,f_m\}$ is a minimal system of 
homogeneous generators for $I$.  Let $n_p$ be the minimal number of 
homogeneous generators of 
$\langle f_1,\ldots,f_{p-1}\rangle: \left\langle f_p \right\rangle$ 
for $p=1,\ldots,m$. Then
\[
\beta_{i,i+j}(I)=\sum_{1 \leq p \leq m,\text{ }{\rm  deg}(f_{p})=j} \binom{n_{p}}{i}.
\]
\end{theorem}

Fix a monomial order $<$ on the ring
$S$. If $I$ is any homogeneous ideal,
then the {\bf initial ideal} of $I$, denoted ${\rm in}_<(I)$ 
(or ${\rm in}(I)$ if $<$ is clear), is the monomial ideal generated 
by the initial terms of the elements of $I$.  We shall need the following 
result:

\begin{theorem}[{\cite[Theorem 22.9]{P}}]\label{in-bound} Fix any monomial order $<$ on $S$.
If $I \subseteq S$ is a homogeneous ideal,  then $\beta_{i,j}(I) \leq \beta_{i,j}({\rm in}_<(I))$
for all $i,j \geq 0$.
\end{theorem}  

For a homogeneous ideal $I \subseteq S$, the
\textbf{Hilbert series} of $S/I$ is the generating function
\[HS_{S/I}(t)=\sum_{i=0}^{\infty} 
\dim_k(S/I)_i t^{i}=\frac{Q(t)}{(1-t)^{d}}
\]
where $Q(t)\in{\mathbb{Z}[t]}$, $d= \dim S/I$, and
$\dim_k(S/I)_i$ is the dimension of the $i$-th graded piece
of $S/I$.  If $Q(t) = h_0 + h_1t + \cdots + h_rt^r$, then the 
\textbf{$h$-vector} is of $S/I$ is $h = (h_0,\ldots,h_r)$.
The next lemma allows to deduce when $I$ 
and ${\rm in}_<(I)$ have the same graded Betti numbers.

\begin{lemma} 
\label{adam-lemma}
Fix any monomial order $<$ on $S$.
Let $I \subseteq S$ be a homogeneous ideal, and let $J = {\rm in}_<(I)$.
Suppose that there exists an integer $k$ such that 
$\beta_{i,i+j}(I) = \beta_{i,i+j}(J)$ for all $i$ and all $j \neq k$. 
Then we also have $\beta_{i,i+k}(I) = \beta_{i,i+k}(J)$ for all $i \geq 0$.
\end{lemma}

\begin{proof}
Let $b_{i,i+j} = \beta_{i,i+j}(I)$ and $c_{i,i+j} = \beta_{i,i+j}(J)$.  
The ideals $I$ and $J$ have the same Hilbert series
(see \cite[Proposition 4.29]{EH}),  which can be expressed 
in terms of their graded 
Betti numbers via the formulas (see \cite[Proposition 4.27]{EH}):
\begin{eqnarray*}
HS_{S/I}(t) = \frac{\sum_{i}(-1)^i\left(\sum_j b_{i,i+j}t^{i+j}\right)}{(1-t)^n} \text{, and } HS_{S/J}(t) = \frac{\sum_{i}(-1)^i\left(\sum_j c_{i,i+j}t^{i+j}\right)}{(1-t)^n} .
\end{eqnarray*}
Our hypotheses imply that $b_{i,i+j} = c_{i,i+j}$ for all $i$ and $j$ 
if $j \neq k$.  Since $HS_{R/I}(t) = HS_{R/J}(t)$,
after this substitution and simplifying, we arrive at
\[\sum_i (-1)^i b_{i,i+k}t^{i+k} = \sum_{i} (-1)^i c_{i,i+k}t^{i+k}. \]
Comparing the coefficients on both sides gives the desired identity.
\end{proof}

%%%%%%%%%%%%%%%%%%%%%%%%%%%%%%%%%%%%%%%%%%%%%%%%%%%%%%%%%%%%%%%%%

\section{The family of graphs $G_{r,d}$ and their toric ideals}
\label{betti-numbers-Grd}

We now define the family of graphs $G_{r,d}$ which will be the focus of our
case study.   Given positive integers $r \geq 3$ and $d \geq 2$,
the graph $G_{r,d}$ has  vertex set
\[
V(G_{r,d})=\{x_{1},x_{2},y_{1},\ldots,y_{d},z_{1},\ldots ,z_{2r-3}\}
\]
and edge set
\begin{align*}
E(G_{r,d})=\{& \{x_{i},y_{j}\} \mid 1 \leq i \leq 2, 1 \leq j \leq d \} \cup \{ \{x_{1},z_{1}\}, \{x_{2},z_{2r-3}\} \}\\
&\cup \{ \{z_{i},z_{i+1}\} \mid  1 \leq i \leq 2r-4 \}.
\end{align*} 
We label the edges of $G_{r,d}$ as follows. Let $e_{1}=\{x_{1},z_{1}\}$, $e_{2r-2}=\{z_{2r-3},x_{2}\}$ and for $i\in{\{2,\ldots,2r-3\}}$ 
let $e_{i}=\{z_{i-1},z_{i}\}$. For $i\in{\{1,\ldots,d\}}$ let $a_{i}=\{x_{1},y_{i}\}$ and $b_{i}=\{x_{2},y_{i}\}$.   Informally, the graph $G_{r,d}$ is the
complete bipartite graph $K_{2,d}$ (the graph with
$V(K_{2,d}) = \{x_1,x_2,y_1,\ldots,y_d\}$ and $E(K_{2,d}) = \{\{x_i,y_j\} ~|~
1\leq i \leq 2, 1 \leq j \leq d \})$) where we have added a path of length
$2r-2$ between the two vertices of degree $d$ in $K_{2,d}$.  See
Figure \ref{example_graph} in the introduction for the graph $G_{3,5}$
equipped with our labelling.

\begin{remark}
It is prudent to discuss the boundary cases.  Note that if $r=1,2$ or 
if $d=1$, then we can still define $G_{r,d}$.
If $r=1$, then we are adding a path of length zero to $K_{2,d}$, and thus
$G_{1,d}=K_{2,d}$.  If $r=2$, then $G_{2,d}=K_{2,d+1}$.  Since the graded Betti 
numbers of $I_{K_{2,d}}$ are known (and we will give a new proof) we can assume 
throughout that $r\ge 3$.  Furthermore, if $d=1$, then
$G_{r,1}$ is a cycle of length of $2r$ if $r \geq 2$, or $G_{1,1} = K_{2,1}$.
In the first case, $I_{G_{r,1}}$ is a principal ideal whose
Betti numbers are again known.  In the second case, 
$I_{G_{1,1}}= \langle 0 \rangle$.   We can thus restrict our discussion
to $r \geq 3$ and $d \geq 2$. 
\end{remark}

We can give an explicit description of the primitive walks of $G_{r,d}$.
Note that we say two closed even walks $W = (e_0,\ldots,e_{2s-1})$ and
$W' = (e'_0,\ldots,e'_{2s-1})$ are equivalent up to a {\bf circular permutation}
if there exists an $i$ such that $e_j  = e'_{j+i}$ 
(or $e_j = e'_{(2s-j)+i}$ if the cycle is in the reverse order) for all $j$ where
$j+i$ (or $(2s-j)+i$) is taken modulo $2s$.
 
\begin{lemma}\label{all-walks}
Fix integers $r\ge 3$, $d\ge 2$. Let $W$ be a primitive walk of $G_{r,d}$. Then, up to a circular permutation,
$W$ is of one of the following:
\begin{enumerate}
\item[$(i)$] $W = (a_{i},b_{i},b_{j},a_{j})$ where $1 \leq i < j \leq d$, or 
\item[$(ii)$] $W = (a_i,e_1,e_2,\ldots,e_{2r-2},b_i)$ where $1 \leq i \leq d$.
\end{enumerate}
\end{lemma}

\begin{proof}
Observe that if we delete $x_1$ or $x_2$ from $G_{r,d}$, then we are left with 
a tree. Any closed even walk on a tree is not minimal and cannot be primitive 
by Lemma~\ref{closed-car}.  This means $W$ must pass through $x_1$ and $x_2$. 
Further, $W$ must pass through $x_1$ and $x_2$ exactly once, otherwise $W$ 
can be expressed as two consecutive closed even walks, and so it is not 
primitive by Lemma~\ref{closed-descom}.   
We can determine $W$ by selecting two edges adjacent to $x_1$. Since $W$ is 
minimal, the selected edges must be distinct. There are two cases.

In the first case, we
select edges $\{a_i,a_j\}$ adjacent to $x_1$ where $i\neq j$. 
Then a primitive walk passing through $x_1$ and $x_2$ exactly once is of the 
form $(a_{i},b_{i},b_{j},a_{j})$.
In the second case, we select edges $\{a_i,e_1\}$ adjacent to $x_1$. Then a 
primitive walk passing through $x_1$ and $x_2$ exactly once is of the 
form $(a_i,e_1,e_2,\ldots,e_{2r-2},b_i)$.

This describes the primitive walks of $G_{r,d}$ up to a circular permutation. 
\end{proof}

\begin{corollary}\label{groebnerbasis}
Fix integers $r \geq 3$ and $d \geq 2$.
A Gr\"{o}bner basis for $I_{G_{r,d}}$ with respect to any monomial ordering
is given by
\[ \{a_{i}b_{j}-b_{i}a_{j} \mid 1 \leq i <j \leq d\}
\cup 
\{a_{i}e_{2}e_{4} \cdots e_{2r-2} - b_{i}e_{1}e_{3}e_{5} \cdots e_{2r-3} \mid 1 \leq i \leq d\}.\]
\end{corollary}

\begin{proof}This corollary 
follows from Theorem~\ref{primitive-theorem} and Lemma
\ref{all-walks}.  Note that while Theorem ~\ref{primitive-theorem} 
specifies all primitive walks, if $W$ and $W'$ are two primitive walks 
that are equivalent up to circular permutation, then $f_W \in I_G$
if and only if $f_{W'} \in I_G$.
\end{proof}

\begin{remark} \label{gen-bip-graph}
Note that the set  $\{(a_{i},b_{i},b_{j},a_{j}) ~\mid~ 1 \leq i < j \leq d\}$ 
is the set 
of all primitive walks, up to circular permutation, 
of the complete bipartite graph $K_{2,d}$. 
By Theorem~\ref{primitive-theorem}
the set  $\{a_{i}b_{j}-b_{i}a_{j} \mid 1 \leq i <j \leq d\}$
is a Gr\"obner basis of the toric ideal $I_{K_{2,d}}$ 
with respect to any monomial order. 
\end{remark}

Let $>$ denote the graded reverse lex order on the polynomial ring 
$k[a_1,\ldots,a_d,$ $b_1,\ldots,b_d,$ $e_1,\ldots,e_{2r-2}]$ where the 
variables are ordered in the following way:  
\[
a_1 > \dots > a_d > e_1 > \dots > e_{2r-2} >b_1 > \dots > b_d.
\]
Going forward, we let ${\rm in}(I_{G_{r,d}})$ denote the initial ideal of 
$I_{G_{r,d}}$  with respect to $>$.  Let $\mathcal{G}_{r,d}$ denote the unique 
minimal generating set of ${\rm in}(I_{G_{r,d}})$.  It follows
from Corollary \ref{groebnerbasis} that 
$\mathcal{G}_{r,d} = F_{d} \cup H_{r,d}$ where
\[
F_{d}=\{a_{i}b_{j} \mid 1 \leq j < i \leq d \}\text{ and }H_{r,d}=\{a_{i}e_{2}e_{4} \cdots e_{2r-2} \mid 1 \leq i \leq d\}.
\]
Each monomial of $F_{d}$ has degree two and each monomial of 
$H_{r,d}$ has degree $r$.  Also
\[\left| \mathcal{G}_{r,d} \right|= |F_d|+|H_{r,d}| = \binom{d}{2} + d = \binom{d+1}{2} = \frac{d(d+1)}{2}.\]

Our immediate goal is to show that ${\rm in}(I_{G_{r,d}})$ has linear
quotients.  We order the generators in $\mathcal{G}_{r,d}$ from 
least to greatest using the graded reverse lex order, and write 
these elements as $m_1, \dots ,m_{\frac{d(d+1)}{2}}$. 
That is, 
\begin{eqnarray*}
F_{d} & =&\{m_1, \dots ,m_{\frac{d(d-1)}{2}}\} \\
& = & \{\underbrace{a_{d}b_{d-1}}_{b_{d-1}\text{ as factor}}, \underbrace{a_{d}b_{d-2}, a_{d-1}b_{d-2},}_{b_{d-2} \text{ as factor}} \underbrace{a_{d}b_{d-3} , a_{d-1}b_{d-3} , a_{d-2}b_{d-3}}_{b_{d-3} \text{ as factor}}, \dots , \underbrace{a_{d}b_{1}, a_{d-1}b_{1}, \dots , a_{2}b_{1}}_{b_1 \text{ as factor}}\},
\end{eqnarray*}
and
\begin{eqnarray*}
H_{r,d} & =& \{m_{{\frac{d(d-1)}{2}}+1}, \dots ,m_{\frac{d(d+1)}{2}}\} = \{a_{d}e_2 \cdots e_{2r-2}, a_{d-1}e_2 \cdots e_{2r-2}, \dots, a_{1}e_2 \cdots e_{2r-2}\}.
\end{eqnarray*}

\begin{example}\label{example-in} We illustrate some of the above
ideas using the graph $G_{3,5}$ (see Figure \ref{example_graph}).
A Gr\"obner basis for $I_{G_{3,5}}$ is given by
\[\{a_{5}b_{4}-b_{5}a_{4},a_{5}b_{3}-b_{5}a_{3},a_{4}b_{3}-b_{4}a_{3},a_{5}b_{2}-b_{5}a_{2},a_{4}b_{2}-b_{4}a_{2},
\]
\[\hspace{.25cm} a_{3}b_{2}-b_{3}a_{2},a_{5}b_{1}-b_{5}a_{1},a_{4}b_{1}-b_{4}a_{1}, 
a_{3}b_{1}-b_{3}a_{1},a_{2}b_{1}-b_{2}a_{1},\]
\[
\hspace{.25cm} a_5e_2e_4 - b_5e_1e_3, a_4e_2e_4 - b_4e_1e_3, a_3e_2e_4 - b_3e_1e_3, a_2e_2e_4 - b_2e_1e_3, a_1e_2e_4 - b_1e_1e_3\}.\]
\normalsize
Using the graded reverse lex order, ${\rm in}(I_{G_{3,5}})$
is generated by
the elements of $\mathcal{G}_{3,5}$:
\[
\{a_5 b_4, a_5 b_3, a_4 b_3, a_5 b_2, a_4 b_2, a_3 b_2, a_5 b_1,
     a_4 b_1, a_3 b_1, a_2 b_1, a_5 e_2 e_4, a_4 e_2 e_4, a_3 e_2 e_4,
     a_2 e_2 e_4, a_1 e_2 e_4\},
\]
which we have ordered from smallest to largest.
\end{example}

\begin{theorem}
\label{quotient-prop}
Fix integers $r \geq 3$ and $d \geq 2$.
Then  ${\rm in}(I_{G_{r,d}})$ has linear quotients with respect to  
$\{m_1, \dots ,m_{\frac{d(d+1)}{2}}\}$.  Furthermore
\begin{align*}
 (n_1, \dots ,n_{\frac{d(d+1)}{2}})= (0,\underbrace{1,1}_{2},\underbrace{2,2,2}_{3},\ldots, \underbrace{d-2,d-2,\dots,d-2}_{d-1},\underbrace{d-1,d-1,\ldots,d-1}_{d})
\end{align*} 
where $n_p$ is the number of minimal number of generators of
$\langle m_1,\ldots,m_{p-1} \rangle : \langle m_p \rangle$.
\end{theorem}

\begin{proof}
Let $I(p) = \left\langle m_{1},\ldots,m_{p-1}\right\rangle : 
\left\langle  m_{p} \right\rangle$, for 
$p\in {\{2,\ldots, \frac{d(d+1)}{2}\}}$.  A generating set of $I(p)$ is 
given by:
\begin{align*}
I(p) = \left\langle \frac{LCM(m_1,m_p)}{m_p}, \frac{LCM(m_2,m_p)}{m_p}, \ldots, \frac{LCM(m_{p-1},m_p)}{m_p}\right\rangle.
\end{align*}
We show that this generating set consists of linear forms. We consider four 
cases depending on the form of $m_p$.

\noindent Case 1: If $m_p = a_{d}b_{j}$, then
\[
I(p)=\langle a_{d}b_{d-1}, \dots, a_{d}b_{j+1}, \dots, a_{j}b_{j+1}\rangle: \langle a_{d}b_{j} \rangle=\langle b_{d-1}, \dots, b_{j+1}\rangle.
\]

\noindent 
Case 2: If $m_p = a_{j+1}b_j$ with $j+1 < d$, then
\[
I(p)=\langle a_{d}b_{d-1}, a_{d}b_{d-2}, a_{d-1}b_{d-2}, \dots , a_{d}b_{j}, \dots , a_{j+2}b_{j} \rangle: \langle a_{j+1}b_j \rangle =\langle a_{j+2}, \dots, a_d\rangle.
\]

\noindent 
Case 3: If $m_p = a_{i}b_{j}$ with $i - j \geq 2$ and $i < d$, then
\begin{align*}
I(p)=&\langle a_{d}b_{d-1}, a_{d}b_{d-2}, \dots, a_{d}b_{j}, 
\dots, a_{i+1}b_{j}\rangle: \langle a_{i}b_{j} \rangle =\langle a_d, \dots , a_{i+1}, b_{i-1}, \dots , b_{j+1}\rangle.
\end{align*}

\noindent 
Case 4: If $m_p = a_ie_2e_4\cdots e_{2r-2}$ for $1\leq i \leq d$, then
\begin{align*}
I(p)=&\langle a_{d}b_{d-1}, a_{d}b_{d-2}, a_{d-1}b_{d-2}, \dots a_{d}b_{1}, a_{d-1}b_{1}, \dots , a_{2}b_{1},\\
& a_de_2e_4\cdots e_{n-2}, \dots , a_{i+1}e_2e_4\cdots e_{2r-2}\rangle: \langle a_i e_2e_4\cdots e_{2r-2} \rangle\\
=&  \begin{cases}
    \langle a_d, \dots , a_2\rangle & \text{for } i = 1 \\
    \langle a_d, \dots , a_{i+1}, b_{i-1}, \dots , b_1\rangle & \text{for } 1 < i < d \\
    \langle b_{d-1}, b_{d-2}, \dots , b_1\rangle  & \text{for } i = d
  \end{cases}
\end{align*}
From the four cases, we see that ${\rm  in}(I_{G_{r,d}})$ has linear quotients.

Let $n_p$ be the number of minimal generators for $I(p)$ as computed above. 
In Cases 1-3, we have $n_p$ = $d - j - 1$ where $j$
is the subscript of the $b$-variable, and in Case 4, we have $n_p$ = $d-1$. 
From our order of $m_1, \dots ,m_{\frac{d(d-1)}{2}}$, the first $\frac{d(d-1)}{2}$
generators,
the subscript $j$ of the $b$-variable follows the order
\begin{align*}
 d-1,\underbrace{d-2,d-2}_{2},\underbrace{d-3,d-3,d-3}_{3},\ldots,\underbrace{1,1,1,\ldots,1}_{d-1}.
\end{align*}
The statement now follows.
\end{proof}

\begin{theorem}
\label{betti-numb-initial-ideal}
Fix integers $r \geq 3$ and $d \geq 2$.
The graded Betti numbers  
of ${\rm in}(I_{G_{r,d}})$ are: 
\begin{eqnarray*}
\beta_{i,i+2}({\rm in}(I_{G_{r,d}})) & =& 
(i+1)\binom{d}{i+2}~~\mbox{for $0 \leq i \leq d-2$}, \\
\beta_{i,i+r}({\rm in}(I_{G_{r,d}})) & =& d\binom{d-1}{i}~~\mbox{for $0 \leq i \leq d-1$}, 
\end{eqnarray*}
and $\beta_{i,i+j}({\rm in}(I_{G_{r,d}}))=0$ everywhere else.  
\end{theorem}
\begin{proof}
Theorem~\ref{quotient-prop} shows ${\rm in}(I_{G_{r,d}})$ has linear quotients 
with respect to the ordered set $\{m_1, \ldots, m_{\frac{d(d+1)}{2}}\}$. By Theorem ~\ref{quotient-compute}, 
$\beta_{i,i+j}({\rm in}(I_{G_{r,d}})) \neq 0$ only if $j\in{\{2,r\}}$.

We now consider the case $j=2$.   By Theorem \ref{quotient-prop} 
the first $\frac{d(d-1)}{2}$ degree two generators satisfy 
$(n_1, \dots ,n_{\frac{d(d-1)}{2}})= 
(0,1,1,2,2,2,\ldots, \underbrace{d-2,d-2,\dots,d-2}_{d-1})$. 
By Theorem ~\ref{quotient-compute} it follows that:
\begin{align*}
\beta_{i,i+2}({\rm in}(I_{G_{r,d}}))&=\sum_{p=1}^{\frac{d(d-1)}{2}} \binom{n_{p}}{i} =\binom{0}{i} + 2\binom{1}{i} + 3\binom{2}{i} + \ldots + (d-1)\binom{d-2}{i}\\
&=\sum_{q=1}^{d-1} q\binom{q-1}{i} = (i+1)\sum_{q=1}^{d-1} \binom{q}{i+1}
 = (i+1)\binom{d}{i+2}.
\end{align*}
The second-to-last equality follows from the identity $k\binom{a}{k} = 
a\binom{a-1}{k-1}$, and the last equality follows from the identity 
$\sum_{j=1}^{a}\binom{j}{b}=\binom{a+1}{b+1}$ for $b > 0$.

When  $j=r$, there are $d$ degree $r$ generators each with $n_{p}=d-1$, so 
we obtain:
\[
\beta_{i,i+r}({\rm in}(I_{G_{r,d}}))=\sum_{p={\frac{d(d-1)}{2}}+1}^{{\frac{d(d+1)}{2}}} \binom{n_{p}}{i} = d \binom{d-1}{i}.\] 
This now completes the proof.
\end{proof}

We can now give a new proof of \cite[Theorem 5.1]{BOVT}.

\begin{corollary} \label{betti-bipart}
Let $K_{2,d}$ be the complete bipartite graph
with $d \geq 2$.  Then ${\rm in}(I_{K_{2,d}})$ has linear quotients, and the graded Betti numbers of $I_{K_{2,d}}$ are given by:
\[\beta_{i,i+2}(I_{K_{2,d}})=\beta_{i,i+2}({\rm in}(I_{K_{2,d}}))= 
(i+1) \binom{d}{i+2} ~\mbox{for $0 \leq i \leq d-2$,}\]
and $\beta_{i,i+j}(I_{K_{2,d}})=0$ everywhere else. 
\end{corollary}

\begin{proof}
By Remark \ref{gen-bip-graph}, 
$\{a_ib_j - b_i a_j \mid 1 \leq i < j \leq d\}$ 
 is a Gr\"obner basis of the toric 
ideal $I_{K_{2,d}}$ for any monomial order, so that 
${\rm in}(I_{K_{2,d}})=\left\langle F_{d} \right\rangle$.
By repeating the arguments of Theorems \ref{quotient-prop} and 
\ref{betti-numb-initial-ideal}, using only the members of $F_{d}$, 
${\rm in}(I_{K_{2,d}})$ has linear quotients and graded 
Betti numbers:
\[\beta_{i,i+2}({\rm in}(I_{K_{2,d}}))= (i+1) \binom{d}{i+2}
~~\mbox{for $0 \leq i \leq d-2$,}\]
and $\beta_{i,i+k}({\rm in}(I_{K_{2,d}})) = 0$ for all $k \neq 2$.
By Theorem \ref{in-bound}, $\beta_{i,i+k}(I_{K_{2,d}}) = 
\beta_{i,i+k}({\rm in}(I_{K_{2,d}})) = 0$ for all $k \neq 2$.  If we apply
 Lemma \ref{adam-lemma}, we get $\beta_{i,j}(I_{K_{2,d}}) = 
\beta_{i,j}({\rm in}(I_{K_{2,d}}))$ for all $i,j \geq 0$.
\end{proof}

We now come to our main theorem.

\begin{theorem}
\label{betti-full}
Fix integers $r \geq 3$ and $d \geq 2$.
Then 
\[\beta_{i,i+j}(I_{G_{r,d}}) = \beta_{i,i+j}({\rm in}(I_{G_{r,d}})) ~~\mbox{for all 
$i,j \geq 0$}.\] 
In particular, $\beta_{i,i+j}(I_{G_{r,d}})$ can be computed using
Theorem \ref{betti-numb-initial-ideal}.
\end{theorem}

\begin{proof}
By Theorems \ref{in-bound} and \ref{betti-numb-initial-ideal}, we 
have $\beta_{i,i+j}(I_{G_{r,d}}) = \beta_{i,i+j}({\rm in}(I_{G_{r,d}})) = 0$ for
all $j \neq 2,r$.  
By Remark~\ref{gen-bip-graph}, 
$I_{K_{2,d}}= \langle a_ib_j - b_i a_j \mid 1 \leq i < j \leq d \rangle $. 
By Corollary \ref {groebnerbasis}, the generators of $I_{K_{2,d}}$ 
are the same as the degree two generators of $I_{G_{r,d}}$.  Because
both $I_{K_{2,d}}$ and
$I_{G_{r,d}}$ have no linear generators, we have 
$\beta_{i,i+2}(I_{G_{r,d}})=\beta_{i,i+2}(I_{K_{2,d}})$ for all $i \geq 0$, 
i.e., they have the same linear strand.   By 
Theorems~\ref{betti-numb-initial-ideal} and \ref{betti-bipart}, 
$\beta_{i,i+2}({\rm in}(I_{G_{r,d}}))=\beta_{i,i+2}(I_{K_{2,d}})$, which gives 
$\beta_{i,i+2}(I_{G_{r,d}})=\beta_{i,i+2}({\rm in}(I_{G_{r,d}}))$
for all $i \geq 0$.  The conclusion now follows by applying
Lemma \ref{adam-lemma} since we have shown that 
$\beta_{i,i+j}(I_{G_{r,d}}) = \beta_{i,i+j}({\rm in}(I_{G_{r,d}})$ for all $j \neq r$.
\end{proof}

%%%%%%%%%%%%%%%%%%%%%%%%%%%%%%%%%%%%%%%%%%%%%%%%%%%

\section{Consequences}\label{consequences}
In this section we record some consequences of Theorem \ref{betti-full}.

We first deduce some bounds on the regularity and projective dimension
for the toric ideal of any graph $G$ using Theorem \ref{betti-full}.   
The {\bf (Castelnuovo-Mumford) regularity} of any ideal $I$ is defined
to be ${\rm reg}(I) = \max\{j-i ~|~ \beta_{i,j}(I) \neq 0\}$, while
the {\bf projective dimension} is ${\rm pdim}(I) = \max\{i ~|~
\beta_{i,j}(I) \neq 0\}$.  Recall that $H$ is an {\bf induced subgraph}
of $G$ if there exits a subset $W \subseteq V(G)$ such
that $H = (W,E(W))$ where $E(W) = \{ e \in E(G) ~|~ e \subseteq W\}$.

\begin{theorem}\label{invariants}
Let $G$ be a graph, and suppose that $G$
has an induced subgraph of the form $H = G_{r_1,d_1} \sqcup G_{r_2,d_2}
\sqcup \cdots \sqcup G_{r_s,d_s}$, i.e., $s$ disjoint subgraphs of the
form $G_{r,d}$ with $r\ge 3$ and $d\ge 2$.  Then
\begin{enumerate}
\item[$(i)$] ${\rm reg}(I_G) \geq  r_1 + \cdots + r_s -s+1$, and
\item[$(ii)$] ${\rm pdim}(I_G) \geq  d_1+\cdots +d_s -1$.
\end{enumerate}
\end{theorem}

\begin{proof}
By Theorem \ref{betti-full}, ${\rm reg}(I_{G_{r,d}}) = r$
and ${\rm pdim}(I_{G_{r,d}}) = d-1$.  Now apply \cite[Theorem 3.7]{BHO}.
\end{proof}

Next, we show that $R/I_{G_{r,d}}$ is Cohen-Macaulay.

\begin{theorem} 
\label{CM}
Fix integers $r \geq 3$ and $d \geq 2$, and let $R = k[E(G_{r,d})]$.
Then $R/I_{G_{r,d}}$ is a Cohen-Macaulay ring with Krull dimension 
${\rm dim}(R/I_{G_{r,d}}) = d + 2r -2$.
\end{theorem}

\begin{proof}
By Corollary \ref{groebnerbasis} the initial ideal of $I_{G_{r,d}}$
with respect to any monomial order is a square-free monomial ideal. 
A result of Sturmfels \cite[Proposition 13.15]{S} implies that 
$R/I_{G_{r,d}}$ is normal, so thus by a result of Hochster 
(see \cite[10]{H} or \cite[Theorem 5.17]{EH}), $R/I_{G_{r,d}}$ is a 
Cohen-Macaulay ring.

Since $R/I_{G_{r,d}}$ is Cohen-Macaulay, the 
Auslander-Buchsbaum formula \cite[Section 1.3]{BH} implies
$
{\rm dim}(R/I_{G_{r,d}}) = {\rm dim}(R) - {\rm pdim}(R/I_{G_{r,d}}).$
Here the Krull dimension of $R$ equals the number of edges in $G_{r,d}$, 
i.e., ${\rm dim}(R)=2d + 2r - 2$.  
Since ${\rm pdim}(R/I_{G_{r,d}}) =d$ by Theorem \ref{invariants},
the conclusion now holds.
\end{proof}

We can also use Theorem \ref{betti-full} to compute the Hilbert
series of $k[E(G_{r,d})]/I_{G_{r,d}}$.

\begin{theorem}
\label{h-vector-theorem}
Fix integers $r \geq 3$ and $d \geq 2$, and let $R = k[E(G_{r,d})]$.
Then the Hilbert series of $R/I_{G_{r,d}}$ is
\begin{eqnarray*}
HS_{R/I_{G_{r,d}}}(t) &= &\frac{1 + dt + dt^2 + \cdots + dt^{r-1}}{(1-t)^{d+2r-2}}.
\end{eqnarray*}
In particular, the h-vector of $R/I_{G_{r,d}}$ is $(1,d,\ldots,d)$.
\end{theorem}

\begin{proof}
Recall that $\beta_{i+1,j}(R/I_{G_{r,d}}) = \beta_{i,j}(I_{G_{r,d}})$
for all $i, j\geq 0$.  In particular, Theorem \ref{betti-full},
also gives the graded Betti numbers of the quotient $R/I_{G_{r,d}}$,
where $\beta_{0,0}(R/I_{G_{r,d}})= 1$.   
  
By \cite[Proposition 4.27]{EH} the Hilbert series satisfies
\begin{eqnarray*}
HS_{R/I_{G_{r,d}}}(t) &= &\frac{\sum_{i}(-1)^i\left(\sum_j \beta_{i,j}t^{j}\right)}{(1-t)^{2d+2r-2}} 
\end{eqnarray*}
where $\beta_{i,j}=\beta_{i,j}({R/I_{G_{r,d}}})$.
Because ${\rm dim}(R/I_{G_{r,d}}) = d + 2r - 2$ by Theorem \ref{CM}, we must
divide out $(1-t)^d$ from the numerator to write the Hilbert series
in lowest terms.  Expanding the expression in the numerator yields:
\begin{align*}
\sum_{i}(-1)^i\left(\sum_j \beta_{i,j}t^{j}\right) &= 1 + \sum^{d-1}_{i=1}(-1)^{i}\cdot i\cdot\binom{d}{i+1}t^{i+1}+\sum^{d}_{i=1}(-1)^{i}\cdot d\cdot\binom{d-1}{i-1}t^{i+r-1} \\
&= 1 + \sum_{i=1}^{d}(-1)^{i+1}(i-1)\binom{d}{i}t^{i} - dt^{r}(1-t)^{d-1}
\end{align*}
where we applied an index shift in the first sum. 
If we now apply the identity ($i-1)\binom{d}{i}=(d-1)\binom{d-1}{i-1} - \binom{d-1}{i}$ we get
\[1 + (d-1)\sum_{i=1}^{d}\left((-1)^{i+1}\binom{d-1}{i-1}t^i\right)- \sum_{i=1}^{d}\left((-1)^{i+1}\binom{d-1}{i}t^i\right) - dt^{r}(1-t)^{d-1}\]
\begin{align*}
&= 1 + (d-1)(1-t)^{d-1}t + (1-t)^{d-1} - 1 - dt^{r}(1-t)^{d-1}\\
&= (1-t)^{d-1}[1 + (d-1)t - dt^r]\\
&= (1-t)^{d-1}[(1-t^r) + t(d-1)(1-t^{r-1})]\\
&= (1-t)^{d-1}[(1-t)(1+\cdots+t^{r-1}) + t(d-1)(1-t)(1+\cdots+t^{r-2})]\\
&= (1-t)^d[1 + dt + dt^2 + \cdots + dt^{r-1}].
\end{align*}
The conclusion follows.
\end{proof}

\begin{remark}\label{rem:Stanley-conj}
Stanley \cite[Conjecture 4]{St} conjectured that if
$S$ is a graded Cohen-Macaulay domain, then the $h$-vector of $S$ was
unimodal, i.e., there exists an index $j$ such that 
$h_i \leq h_{i+1}$ for all $0 \leq i < j$ and $h_i \geq h_{i+1}$ for
$i \geq j$.  Since toric ideals are always prime, Theorem \ref{CM} implies
that $K[E(G_{r,d})]/I_{G_{r,d}}$ is a graded Cohen-Macaulay domain,
and Theorem \ref{h-vector-theorem} shows that these domains satisfy
Stanley's conjecture.  Note that there is some ambiguity in the literature
regarding the precise statement and status of Stanley's conjecture; 
in particular, see \cite{M} and \cite{B}, but also refer to the appendix in the 
corrected version of \cite{B} on the arXiv ({\tt arXiv:1505.07377v3} [math.CO]). 
\end{remark}

%%%%%%%%%%%%%%%%%%%%%%%%%%%%%%%%%%%%%%%%%%%%%%%%%%%%%%%%%%%%%%%%%
\bibliographystyle{plain}

\end{document}